\documentclass[11pt]{article}
\usepackage{layout}
\usepackage[makeroom]{cancel}
\usepackage{bbm}
\usepackage{mathtools}
\mathtoolsset{showonlyrefs,showmanualtags}
\usepackage{amsfonts,dsfont}
\usepackage{amsmath}
\usepackage{amssymb,bm}
\usepackage{amsthm}
\usepackage{graphicx} 
\usepackage{multicol}
\usepackage{mathrsfs}
\usepackage{enumerate,enumitem}
\usepackage[rgb,dvipsnames]{xcolor}
\usepackage{float}
\usepackage[normalem]{ulem}
\usepackage{circuitikz}
\usepackage{cite}
\usepackage[colorlinks,linkcolor=blue,citecolor=blue,pagebackref,hypertexnames=false, breaklinks]{hyperref}

\usepackage{float}

\newtheorem{theorem}{Theorem}[section]

\newtheorem{assumption}[theorem]{Assumption}
\newtheorem{definition}[theorem]{Definition}
\newtheorem{proposition}[theorem]{Proposition}

\newtheorem{corollary}[theorem]{Corollary}
\newtheorem{lemma}[theorem]{Lemma}
\newtheorem{remark}[theorem]{Remark}
\newtheorem{example}[theorem]{Example}
\newtheorem{examples}[theorem]{Examples}

\newtheorem{open question}[theorem]{Open Question}
\newtheorem{c/p}[theorem]{Conjecture/Proposition}

\newcommand{\brak}[1]{\left(#1\right)} 

\def\vint{\mathop{\mathchoice%
 {\setbox0\hbox{$\displaystyle\intop$}\kern 0.22\wd0%
 \vcenter{\hrule width 0.6\wd0}\kern -0.82\wd0}%
 {\setbox0\hbox{$\textstyle\intop$}\kern 0.2\wd0%
 \vcenter{\hrule width 0.6\wd0}\kern -0.8\wd0}%
 {\setbox0\hbox{$\scriptstyle\intop$}\kern 0.2\wd0%
 \vcenter{\hrule width 0.6\wd0}\kern -0.8\wd0}%
 {\setbox0\hbox{$\scriptscriptstyle\intop$}\kern 0.2\wd0%
 \vcenter{\hrule width 0.6\wd0}\kern -0.8\wd0}}%
 \mathopen{}\int}

\newcommand{\DF}{\mathcal{E}}

\newcommand{\B}{\mathbf B}

\usepackage[textwidth=3.2cm]{todonotes}

\title{BV functions and fractional Laplacians on Dirichlet spaces}

\author{Patricia Alonso Ruiz\footnote{P.A.R. was partly supported by the grant DMS~\#1855349 of the NSF (U.S.A.).},  
Fabrice Baudoin\footnote{F.B. was partly supported by the grants DMS~\#1660031, ~\#1901315 of the NSF (U.S.A.) and a Simons Foundation Collaboration grant. }, 
Li Chen, 
Luke Rogers\footnote{L.R. was partly supported by the grant DMS~1659643 of the NSF (U.S.A.). }, \\ 
Nageswari Shanmugalingam\footnote{N.S. was partly supported by the grant DMS~\#1800161 of the NSF (U.S.A.). }, 
Alexander Teplyaev\footnote{A.T. was partly supported by the grant DMS~\#1613025 of the NSF (U.S.A.). }}  

\date{\today}

\begin{document}

\maketitle

\begin{abstract}
We study $L^p$ Besov critical exponents and isoperimetric and Sobolev inequalities associated with fractional Laplacians on metric measure spaces. The main tool is the theory of heat semigroup based Besov  classes in Dirichlet spaces that was introduced by the authors in previous works.
\end{abstract}

\tableofcontents

\section{Introduction}
The authors presented, in~\cite{ABCRST3}, a heat kernel approach to the theory of functions of bounded variation (BV) in Dirichlet spaces with sub-Gaussian heat kernel bounds.  In the present work those results are extended to the non-local setting using a fractional power of the Laplacian.  We believe that this approach may admit further connections to research in PDE theory. In particular, the study of fractional Sobolev spaces and BV functions in different levels of 
generality 
is currently the subject of  extensive research, see e.g.~\cite{Zue19,KM19,MP19} for domains with fractal boundary, manifolds and Carnot groups, and also the distributional approach recently introduced in~\cite{CS19}. We also refer to \cite{2019arXiv190702281G} and the references therein.

\medskip

An easily studied class of non-local Dirichlet forms are those whose associated semigroup is obtained via subordination and have as generator a fractional Laplacian $(-L)^\delta$, where $L$ is the original (negative definite) Laplacian, see Section~\ref{S:Preliminaries}. For simplicity of  presentation and in order to use results previously proved in \cite{ABCRST3}, the present paper focuses on that particular case, though some of our results could be extended to a slightly more general setting. The definition of Besov spaces $\mathbf{B}^{p,\alpha}(X)$ based on the heat semigroup in this setting follows~\cite{ABCRST1,ABCRST2,ABCRST3} and is based on a generalization of the heat semigroup characterization of BV functions on $\mathbb{R}^d$ introduced in~\cite{MPPP07}, whose ideas go back to the work of de Giorgi~\cite{DeG54} and Ledoux~\cite{Led94}. For any $\alpha>0$ and $1\leq p<\infty$, the space $\mathbf{B}^{p,\alpha}(X)$ is the collection of functions $f\in L^p(X,\mu)$ for which
\begin{equation*}
\| f \|_{p,\alpha}= \sup_{t >0} t^{-\alpha} \left( \int_X  P_t^{(\delta)}(|f-f(x) |^p)(x)\,d\mu(x) \right)^{1/p}<\infty,
\end{equation*}
where $P^{(\delta)}_t $ denotes the heat semigroup associated with the fractional Laplacian $(-L)^\delta$ of order $0<\delta<1$. For general properties of these spaces we refer to~\cite{ABCRST1}; further results in the context of local Dirichlet forms with Gaussian heat kernel estimates can be found in~\cite{ABCRST2}, while the sub-Gaussian case is treated in~\cite{ABCRST3}. 

\medskip

In the local case, the space $\mathbf{B}^{p,\alpha}(X)$ can be identified with a space of Korevaar-Schoen type~\cite[Theorem 2.4]{ABCRST3}. In the non-local setting discussed in this paper this is true for $1\le p\le \infty$ and $\alpha<1/p$, see Theorem~\ref{T:KS-Besov}, while for $\alpha>1/p$ the spaces are trivial, c.f.\ Proposition~\ref{trivial Lp}. Note that the threshold $\alpha=1/p$ reflects only the non-locality and is independent of the parameter $\delta$ of the fractional Laplacian.

\medskip

Proposition~\ref{P:W_norm} shows that the case $\alpha=1/p$ provides a natural candidate for characterizing fractional Sobolev spaces in metric measure spaces via heat semigroups. This characterization coincides with Gagliardo's classical definition~\cite{Gag57} on $\mathbb{R}^d$, given by
\begin{equation*}
W^{s,p}(\Omega):=\bigg\{f\in L^p(\Omega,dx)\;\colon\; \bigg(\int_\Omega\int_\Omega\frac{|f(x)-f(y)|^p}{|x-y|^{d+sp}}dy\,dx\bigg)^{1/p}<\infty\bigg\}
\end{equation*}
for $0<s<1$, $1\leq p<\infty$ and $\Omega\subseteq \mathbb{R}^d$.

We use our heat semigroup approach to 
define and study the space of BV functions, which arises as the space $\mathbf{B}^{1,\alpha_1^\#}(X)$ at a critical parameter $\alpha_1^\#$, the latter being defined as
\[
\alpha_1^\#(X):=\sup\{\alpha>0\,\colon\,\mathbf{B}^{\alpha,1}(X)\text{ contains non-constant functions}\}.
\]
Our main results are a co-area formula, an $L^1$-pseudo Poincar\'e inequality and isoperimetric inequalities. An analogous definition of the critical exponents $\alpha_p^\#$ for $p>1$ further allows us to investigate the spaces $\mathbf{B}^{p,\alpha_p^\#}(X)$, which turn out to coincide with fractional Sobolev spaces. Besides proving Sobolev inequalities, we show that these spaces are dense in $L^p$.

An interesting feature of the theory is the following dichotomy: when the space of BV functions is a fractional Sobolev space it is possible to prove all results on BV functions without making any further geometric assumptions, but when the BV space is a Korevaar-Schoen space our proofs need the significant assumption that a \textit{weak Bakry-\'Emery} nonnegative curvature condition of the form
\begin{equation*}
| P^{(\delta)}_t f (x)-P^{(\delta)}_tf(y)| \le C \frac{d(x,y)^\kappa}{t^{\frac{\kappa}{\delta d_W}}} \| f\|_{ L^\infty(X,\mu)},
\end{equation*}
is valid for some suitable $0<\kappa<d_W$, where $d_W>0$ denotes the so-called walk dimension of $X$, see Assumption~\ref{sGKHE} and Lemma~\ref{L:wBECdelta}.  This condition expresses  how the fractional order of the Laplacian affects the regularity of the associated semigroup. The analogous condition in the local setting was introduced in earlier work by the authors~\cite{ABCRST3}.  A particular version of this dichotomy is that if we subordinate a heat flow that satisfies a weak Bakry-\'Emery inequality, then for $\delta>1-\frac{\kappa}{d_W}$ we need the corresponding Bakry-\'Emery inequality to prove results on BV, whereas if $\delta\leq1-\frac{\kappa}{d_W}$ the process is sufficiently non-local that elementary estimates suffice.


\medskip

The structure of the paper is as follows: Section~\ref{S:Preliminaries} sets up the general framework for the spaces under consideration and briefly reviews the construction of fractional Laplacians and their corresponding semigroups via subordination. In Section~\ref{S:Besov}, we introduce the Besov spaces $\mathbf{B}^{p,\alpha}(X)$ and classify them according to the three cases $\alpha<1/p$, $\alpha=1/p$ and $\alpha>1/p$. The Bakry-\'Emery nonnegative curvature and the dichotomy between the Korevaar-Schoen and fractional Sobolev space settings appears in Section~\ref{S:BVandFL}, where BV functions and related functional inequalities are discussed. Section~\ref{S:fractional_Lp} concerns the $L^p$ theory for $p>1$ and relates the spaces $\mathbf{B}^{p,\alpha}(X)$ at criticality with classical fractional Sobolev spaces. The paper concludes in Section~\ref{S:final} with an overview of the general isoperimetric and Sobolev inequalities available in the non-local setting.

%


\section{Preliminaries: Fractional Laplacians on metric measure spaces}\label{S:Preliminaries}

\subsection{Standing assumptions}
Throughout the paper, $(X,d,\mu)$ will denote a locally compact metric measure space, where $\mu$ is a Radon measure supported on $X$. It will be equipped with a Dirichlet form $(\mathcal{E},\mathcal{F}={\rm dom}(\mathcal{E}))$, that is: a densely defined, closed, symmetric and Markovian bilinear form on $L^2(X,\mu)$, see~\cite{FOT94,CF12}. The vector space of continuous functions with compact support in $X$ is denoted $C_c(X)$ and $C_0(X)$ is its closure with respect to the supremum norm. Recall, see for example~\cite[p.6]{FOT94}, that a Dirichlet form $(\mathcal{E},\mathcal{F})$ is called regular if it admits a core, which is a subset of $C_c(X) \cap \mathcal{F}$ that is dense in $C_c(X)$ in the supremum norm and dense in $\mathcal{F}$ in the norm 
\[
\|f\|_{\DF_1}:= \left( \| f \|_{L^2(X,\mu)}^2 + \mathcal{E}(f,f) \right)^{1/2}.
\]
Also, $(\mathcal{E},\mathcal{F})$ is said to be strongly local if $\mathcal{E}(f,g)=0$ for any two functions $f,g\in\mathcal{F}$ with compact supports such that $f$ is constant in a neighborhood of the support of $g$. The heat semigroup associated with the Dirichlet form $(\mathcal{E},\mathcal{F})$ is denoted $\{P_t\}_{t>0}$, and its associated infinitesimal generator is $L$.
We make the following assumptions on the Dirichlet space $(X,\mathcal{E},\mathcal{F}, \mu)$.
\begin{assumption}[Regularity]\label{A1}\ 
\begin{enumerate}[leftmargin=2.5em,label={\rm (A\arabic*)}]
\item For any $x\in X$ and $r>0$, $B(x,r):=\{y\in X\mid d(x,y)<r\}$ has compact closure;
\item The space $(X,d,\mu)$ is Ahlfors $d_H$-regular, i.e.\ there exist $c_{1},c_{2},d_{H}>0$ such that 
\[
c_{1}r^{d_{H}}\leq\mu\bigl(B(x,r)\bigr)\leq c_{2}r^{d_{H}}\qquad \forall\,r\geq 0.
\]
\item The Dirichlet form $(\mathcal{E},\mathcal{F})$ is regular and strongly local.
\end{enumerate}
\end{assumption}
\begin{assumption}[Sub-Gaussian Heat Kernel Estimates]\label{sGKHE}
The semigroup $\{P_t\}_{t>0}$ has a continuous heat kernel $p_t(x,y)$ satisfying, for some $c_{3},c_{4}, c_5, c_6 \in(0,\infty)$ and $d_{W}\in [2,+\infty)$,
 \begin{equation}\label{eq:subGauss-upper}
 c_{5}t^{-d_{H}/d_{W}}\exp\biggl(-c_{6}\Bigl(\frac{d(x,y)^{d_{W}}}{t}\Bigr)^{\frac{1}{d_{W}-1}}\biggr) 
 \le p_{t}(x,y)\leq c_{3}t^{-d_{H}/d_{W}}\exp\biggl(-c_{4}\Bigl(\frac{d(x,y)^{d_{W}}}{t}\Bigr)^{\frac{1}{d_{W}-1}}\biggr)
 \end{equation}
 for $\mu\!\times\!\mu$-a.e.\ $(x,y)\in X\times X$ and each $t\in\bigl(0,+\infty\bigr)$. 
\end{assumption}

The parameter $d_H$ is the Hausdorff dimension, and the parameter $d_W$ is usually called the walk dimension of the space.
Under these assumptions, the Dirichlet form admits the expression
\[
\mathcal{E}(f,f) \simeq \limsup_{r\to 0^+}\int_X\int_{B(x,r)}\frac{|f(y)-f(x)|^2}{r^{d_W+d_H}} \, d\mu(y)\, d\mu(x)
\]
for any $f \in \mathcal{F}$, c.f.~\cite[Section 5.3]{Gri03} and \cite[Theorem 3.3.1]{Baudoin20}. In general, we write $a\simeq b$ if there exist $c_1,c_2>0$ such that $c_1a\leq b\leq c_2a$.

\medskip

In the same spirit as~\cite{ABCRST3}, the definition of fractional Sobolev spaces introduced in Section~\ref{S:Besov} relies on the following class of Korevaar-Schoen spaces. The classical definitions can be found in~\cite{KS93,KST04}. For any $\lambda>0$ and $1\leq p<\infty$ one defines 
\begin{equation}\label{E:def_KS}
KS^{\lambda, p}(X):=\{f\in L^p(X,\mu)\,\colon\,\|f\|_{KS^{\lambda p}(X)}<\infty\},
\end{equation}
where
\begin{equation}\label{E:def_KS_norm}
\|f\|_{KS^{\lambda, p}(X)}^p:=\limsup_{r\to 0^+}\int_X\int_{B(x,r)}\frac{|f(x)-f(y)|^p}{r^{\lambda p}\mu(B(x,r))}d\mu(y)\,d\mu(x).
\end{equation}
Moreover, we define the space $\mathcal{KS}^{\lambda,p}(X)$ analogously to~\eqref{E:def_KS} but with the semi-norm
\begin{equation}\label{E:def_calKS_norm}
\|f\|_{\mathcal{KS}^{\lambda, p}(X)}^p:=\sup_{r>0}\int_X\int_{B(x,r)}\frac{|f(x)-f(y)|^p}{r^{\lambda p}\mu(B(x,r))}d\mu(y)\,d\mu(x).
\end{equation}

The following weak Bakry-\'Emery type estimate  will be essential in our discussion of BV functions. Some of the results proved later, in particular those in Section~\ref{S:Besov}, will not require this condition, however we will assume it throughout the paper for ease of presentation.

\begin{assumption}[Weak Bakry-\'Emery  estimate]\label{A3}\
There exists a constant $C>0$  such that for every $t >0$, $g \in L^\infty(X,\mu)$ and  $x,y \in X$,
\begin{equation}\label{E:wBECD}
| P_t g (x)-P_tg(y)| \le C \frac{d(x,y)^\kappa}{t^{\kappa/d_W}} \| g \|_{ L^\infty(X,\mu)},
\end{equation}
where $\kappa$ is a critical exponent defined by
\[
d_W-\kappa=\sup \{ \lambda >0\, :\, KS^{\lambda,1}(X) \text{ contains non-constant functions}\}.
\]
\end{assumption}

This condition appears in~\cite{ABCRST3} where several examples of Dirichlet spaces satisfying it are studied.

\subsection{Subordination and fractional Laplacians}
Let $0 < \delta < 1$ be fixed throughout the paper. The fractional power of $(-L)^\delta$ can be defined via the following formula, see (5) in~\cite[p. 260]{Yos80},
\begin{equation}\label{As}
(-L)^\delta f = - \frac{\delta}{\Gamma(1-\delta)} \int_0^\infty t^{-\delta-1} (P_t f - f)\ dt.
\end{equation}

It is well-known that $(-L)^\delta$ is the generator of a Markovian semigroup $\{P_t^{(\delta)}\}_{t>0}$ which is related to $\{P_t\}_{t>0}$ by the subordination formula
\begin{equation}\label{eq:Ptdelta}
 P_t^{(\delta)} f(x):=\int_0^\infty \eta^{(\delta)}_t(s) \,P_sf(x)\, ds,
\end{equation}
where $\eta^{(\delta)}_t(s)$ is the non-negative continuous function such that
\begin{equation}\label{E:subordinator_int}
\int_0^\infty\eta^{(\delta)}_t(s)e^{-s\lambda}ds=e^{-t\lambda^\delta}
\end{equation}
for any $\lambda>0$, see for example, Proposition~1 in~\cite[p.260]{Yos80}. In addition, the subordinator $\eta^{(\delta)}_t(s)$ satisfies the following upper bound 
\begin{equation}\label{bound eta}
\eta^{(\delta)}_t(s) \le C \left(\frac{1}{t^{1/\delta}} \wedge \frac{t}{s^{1+\delta}} \right)
\end{equation}
and  for $-\infty < \alpha <\delta$
\begin{equation}\label{moment eta}
\int_0^{+\infty} \eta^{(\delta)}_t(s) s^{\alpha} ds=\frac{\Gamma(1-\alpha/\delta)}{\Gamma(1-\alpha)} t ^{\alpha /\delta}.
\end{equation}
If $\alpha \ge \delta$, then
\[
\int_0^{+\infty} \eta^{(\delta)}_t(s) s^{\alpha} ds=+\infty.
\]
It is known that under Assumptions \ref{A1} and \ref{sGKHE} the Dirichlet form $(\mathcal{E}^{(\delta)},\mathcal{F}^{(\delta)})$ associated with $P_t^{(\delta)}$ satisfies
\begin{equation}\label{E:Edelta}
\mathcal{E}^{(\delta)}(f,f)\simeq \int_X\int_X\frac{|f(x)-f(y)|^2}{d(x,y)^{d_H+\delta d_W}}\,d\mu(y)d\mu(x)\
\end{equation}
and that for the corresponding heat kernel $p_t^{(\delta)} (x,y)$ one has
\def\Phihke#1{\left(1+#1\right)^{-d_H-\delta d_W}}
\begin{equation}\label{eq:HKE-non-loc}
c_{5}t^{-\frac{d_{H}}{\delta d_{W}}}
\Phihke{ c_{6} \frac{d(x,y)}{t^{\frac{1}{\delta d_{W}}}}}
\le 
p_t^{(\delta)} (x,y)
\le
c_{3}t^{-\frac{d_{H}}{\delta d_{W}}}
\Phihke{ c_{4} \frac{d(x,y)}{t^{\frac{1}{\delta d_{W}}}}},
\end{equation}
see, for example, \cite[Lemma 5.4]{Gri03}.
From~\eqref{eq:Ptdelta}, \eqref{moment eta} and Assumption~\ref{A3} one readily obtains the following estimate.
\begin{lemma}\label{L:wBECdelta}
There exists a constant $C>0$  such that for every $t >0$, $g \in L^\infty(X,\mu)$ and  $x,y \in X$,
\begin{equation*}\label{wBECdelta}
| P_t^{(\delta)} g (x)-P_t^{(\delta)} g(y)| \le C \frac{d(x,y)^\kappa}{t^{\frac{\kappa}{\delta d_W}}} \| g \|_{ L^\infty(X,\mu)}.
\end{equation*}
\end{lemma}

\subsection{Examples}
\subsubsection*{Poisson kernel in $\mathbb{R}^d$}
The most classical example of non-local Dirichlet form is that associated with the fractional Laplacian $(-\Delta)^{1/2}$ . The corresponding heat kernel $q_t\colon\mathbb{R}^d\times\mathbb{R}^d\to[0,\infty)$ is given for any $t>0$ by 
\begin{equation*}\label{E:Poisson_HK}
q_t(x,y)=\frac{\Gamma\big(\frac{d+1}{2}\big)\pi^{-\frac{d+1}{2}}}{t^d}\left(1+\frac{|x-y|^2}{t^2}\right)^{-\frac{d+1}{2}}
\end{equation*}
and provides the fundamental solution to the \textit{Poisson equation} in $\mathbb{R}^d$, $\partial_tf= -\Delta^{1/2}f$.

\subsubsection*{Non-local forms on nested fractals}
The concept of fractional metric measure space was introduced by Barlow in~\cite{Bar98}, to which we refer the reader for a precise definition. Some of these spaces support what is called a fractional diffusion, so that both Assumption~\ref{A1} and Assumption~\ref{sGKHE} are valid. Nested fractals like the Vicsek set and the Sierpinski gasket are examples that fall into this class of spaces and also satisfy the weak Bakry-\'Emery estimate from Assumption~\ref{A3}, see~\cite[Theorem 5.1]{ABCRST3}.

\medskip

The non-local Dirichlet form obtained through subordination from such a fractional diffusion has been studied in the literature by many authors, see e.g.~\cite{Sto00,Kum03,CK03,KL16,Kum04,bogdan02,bogdan03}. In particular, it was proved in~\cite{Sto00,Kum03} that for some values of $\delta$, the associated stable-like process whose corresponding Dirichlet form satisfies~\eqref{E:Edelta} can also be obtained as the trace of a $d$-dimensional Brownian motion on $X$, assuming $X\subset\mathbb{R}^d$.

\medskip
 
Notice that, even though the Sierpinski carpet is a fractional metric measure space that admits a unique natural fractional diffusion~\cite{BB99,BBKT},
 the question of whether this space satisfies the weak Bakry-\'Emery estimate with the correct parameter from Assumption~\ref{A3} remains open~\cite[Conjecture 5.4]{ABCRST3}.

\section{Heat kernel based Besov classes for the fractional Laplacian}\label{S:Besov}
This section does not require the weak Bakry-\'Emery condition and could have therefore been written under more general assumptions. However, we will continue working in the same framework for the ease of the presentation. 
As in \cite{ABCRST1,ABCRST2, ABCRST3}, for any $p \ge 1$ and $\alpha \ge 0$ we consider the Besov seminorm
\begin{equation*}\label{E:def_Besov_norm}
\| f \|_{p,\alpha}= \sup_{t >0} t^{-\alpha} \left( \int_X \int_X |f(x)-f(y) |^p p^{(\delta)}_t (x,y) d\mu(x) d\mu(y) \right)^{1/p}
\end{equation*}
and the associated heat semigroup-based class
\[
\mathbf{B}^{p,\alpha}(X)=\{ f \in L^p(X,\mu)\, :\, \| f \|_{p,\alpha} <\infty \}.
\]
We refer to~\cite{ABCRST1} for an account of some of the basic properties of these spaces. The remainder of this section is devoted to identifying and classifying them depending on the relation between the parameters $p$ and $\alpha$.

\subsection{The case $\alpha <1/p$}

We start by comparing the space $\mathbf{B}^{p,\alpha}(X)$ to the Korevaar-Schoen classes defined in \eqref{E:def_KS} and \eqref{E:def_calKS_norm}. 

\begin{theorem}\label{T:KS-Besov}
Let $1\leq p<\infty$ and $0<\alpha < \frac{1}{p}$. Then,
\begin{equation}\label{E:KS-Besov}
\B^{p,\alpha}(X)=\mathcal{KS}^{\alpha \delta d_W\!, p}(X)=KS^{\alpha \delta d_W\!, p}(X).
\end{equation}
Moreover, $\mathbf{B}^{p,\alpha}(X)$ and $\mathcal{KS}^{\alpha \delta d_W,p}(X)$ have equivalent seminorms.
\end{theorem}
\begin{proof}
The inclusion $\mathcal{KS}^{\alpha \delta d_W\!, p}(X)\subseteq KS^{\alpha \delta d_W\!, p}(X)$ follows directly from the definition, while $KS^{\alpha \delta d_W\!, p}(X)\subseteq \mathcal{KS}^{\alpha \delta d_W\!, p}(X)$ is obtained verbatim to~\cite[Proposition 4.1]{ABCRST3}. We now proceed to prove $\B^{p,\alpha}(X)=\mathcal{KS}^{\alpha \delta d_W\!, p}(X)$.

Let us write $\Phi(s)= (1+c_4 s)^{-d_H-\delta d_W}$, so the heat kernel estimate~\eqref{eq:HKE-non-loc} implies $p^{(\delta)}_t (x,y)\geq\Phi(1)t^{\frac{-d_H}{\delta d_W}}$ on $B(y,t^{\frac{1}{\delta d_W}})$. Then,
\begin{align*}
\lefteqn{\frac{1}{t^{\alpha p}}\int_X \int_X |f(x)-f(y) |^p p^{(\delta)}_t (x,y) d\mu(x) d\mu(y)}\quad&\\
&\geq  \frac{1}{t^{\alpha p}} \int_X \int_{B(y,t^{\frac{1}{\delta d_W}})} |f(x)-f(y) |^p p^{(\delta)}_t (x,y) d\mu(x) d\mu(y)\\
&\geq \Phi(1) \int_X \int_{B(y,t^{\frac{1}{\delta d_W}})}  \frac{|f(x)-f(y) |^p}{t^{\alpha p+\frac{d_H}{\delta d_W}}} d\mu(x) d\mu(y),
\end{align*}
and taking the supremum over $t>0$ yields $\mathbf{B}^{p,\alpha}(X)\subseteq\mathcal{KS}^{\alpha\delta d_W\!,p}(X)$.
For the upper bound, fix $r>0$ and set
\begin{align}
A(t,r)&:=\int_X\int_{X\setminus B(y,r)}|f(x)-f(y)|^{p}p^{(\delta)}_{t}(x,y)\,d\mu(x)\,d\mu(y),\label{E:A(t)}\\
B(t,r)&:=\int_X\int_{B(y,r)}|f(x)-f(y)|^{p}p^{(\delta)}_{t}(x,y)\,d\mu(x)\,d\mu(y).\label{E:B(t)}
\end{align}
From the proof of~\cite[Theorem 3.1]{Gri03} we know that~\eqref{eq:HKE-non-loc} implies
\[
\int_{X\setminus B(y,r)}p^{(\delta)}_t(x,y)\,d\mu(x)\leq C\int_{\frac{1}{2}rt^{-\frac{1}{\delta d_W}}}^\infty s^{d_H}\Phi(s)\frac{ds}{s}.
\]
Applying the inequality
$|f(x)-f(y)|^{p}\leq 2^{p-1}(|f(x)|^{p}+|f(y)|^{p})$, the Fubini theorem and the preceding inequality we obtain
\begin{align}
A(t,r)&\leq 2^p\int_X\int_{X\setminus B(y,r)}|f(y)|^pp^{(\delta)}_t(x,y)\,d\mu(y)\,d\mu(x)\notag\\
&\leq 2^pC\|f\|^p_{L^p(X,\mu)}\int_{\frac{1}{2}rt^{-\frac{1}{\delta d_W}}}^\infty s^{d_H}\Phi(s)\frac{ds}{s}\notag\\
&\leq C t^{p\alpha} r^{-\alpha\delta d_W p} \|f\|^p_{L^p(X,\mu)}\int_{\frac{1}{2}rt^{-\frac{1}{\delta d_W}}}^\infty s^{d_H+\alpha\delta d_W p}\Phi(s)\frac{ds}{s}.
\label{eq:HKBesov-norms-upper-proof1}
\end{align}
On the other hand, for $B(t,r)$, writing $r_k=2^{-k}r$ we have by virtue of~\eqref{eq:HKE-non-loc} that
\begin{align}
B(t,r)
&\leq C t^{-\frac{d_H}{\delta d_W}} \sum_{k=0}^\infty\Phi\Bigl(r_{k+1}t^{-\frac{1}{\delta d_W}}\Bigr) r_k^{d_H+\alpha\delta d_W p}\int_X\int_{B(x,r_k)}\frac{|f(x)-f(y)|^p}{r_k^{\alpha\delta d_Wp+d_H}} \,d\mu(x)\,d\mu(y)\notag\\
&\leq Ct^{\alpha p}\|f\|_{\mathcal{KS}^{\alpha\delta d_W,p}(X)}^p\int_0^\infty s^{d_H+\alpha\delta d_W p}\Phi(s)\frac{ds}{s}.\label{eq:HKBesov-norms-upper-proof2}
\end{align}
The integrals in both~\eqref{eq:HKBesov-norms-upper-proof1} and~\eqref{eq:HKBesov-norms-upper-proof2} are  bounded because $\alpha<1/p$ by assumption, see e.g.~\cite[Definition 2.3]{Gri03}, so the bound on $A(t,r)+B(t,r)$ yields
\[
\frac{1}{t^{\alpha p}}\int_X \int_X |f(x)-f(y) |^p p_t^{(\delta)} (x,y) d\mu(x)\, d\mu(y)\leq C_{p,\alpha} \Big(\frac{1}{r^{d_W\alpha\delta p}} \|f\|^p_{L^{p}(X,\mu)}+\|f\|_{\mathcal{KS}^{\alpha\delta d_W,p}(X)}^p \Big)
\]
for some $C_{p,\alpha}>0$ and any $t>0$. Taking the supremum over $t>0$ we obtain $\mathcal{KS}^{\alpha\delta d_W\!,p}(X)\subseteq\mathbf{B}^{p,\alpha}(X)$ and letting $r\to\infty$ gives the equivalence of the seminorms.
\end{proof}

\subsection{The case $\alpha = 1/p$}
When $\alpha = 1/p$, the space $\mathbf{B}^{p,\alpha}(X)$ does not compare to a Korevaar-Schoen class. 
Instead, it will coincide with the \textit{fractional Sobolev space} defined as
\begin{equation*}\label{E:def_W_space}
\mathcal{W}^{\lambda,p}(X):=\left\{ f \in L^p(X,\mu)\,\colon\, W_{\lambda,p}(f) <+\infty \right\},
\end{equation*}
where
\begin{equation*}\label{E:def_W_norm}
W_{\lambda, p}(f):=\Big(\int_X\int_X\frac{|f(x)-f(y)|^p}{d(x,y)^{d_H+\lambda p}}\,d\mu(y)d\mu(x)\Big)^{1/p}.
\end{equation*}

\begin{proposition}\label{P:W_norm}
Let $1\leq p<\infty$. Then
$\B^{p,1/p}(X)=\mathcal{W}^{ \delta d_W/p,p}(X)$ 
with equivalent norms.
\end{proposition}

\begin{proof}
Observe from~\eqref{eq:HKE-non-loc} that 
\begin{equation*}
	c_5\big( t^{\frac{d_H}{\delta d_W}} + c_6 d(x,y)\bigr)^{-d_H-\delta d_W}
	\leq  \frac1t p^{(\delta)}_t(x,y)
	\leq c_3\big( t^{\frac{d_H}{\delta d_W}} + c_4 d(x,y)\bigr)^{-d_H-\delta d_W}.
	\end{equation*}
The upper bound gives
\begin{equation}\label{E:W_norm_01}
\frac{1}{t}\int_X\int_X |f(x)-f(y)|^pp^{(\delta)}_t(x,y)\,d\mu(x)\,d\mu(y)
\leq C\int_X\int_X\frac{|f(x)-f(y)|^p}{d(x,y)^{d_H+\delta d_W}}d\mu(y)\,d\mu(x)=CW_{p,\delta d_W/p}^p(f),
\end{equation}
from which $\|f\|_{p,1/p}\leq CW_{p,\delta d_W/p}^p(f)$. 
The lower bound is
\begin{equation}\label{E:W_norm_02}
\int_X\int_X\frac{|f(x)-f(y)|^p}{(t^{\frac{1}{\delta d_W}}+d(x,y))^{d_H+\delta d_W}}d\mu(y)\,d\mu(x)
\leq \frac{C}{t}\int_X\int_X|f(x)-f(y)|^pp^{(\delta)}_t(x,y)\,d\mu(x)\,d\mu(y)\leq C \|f\|_{p,1/p}^p.
\end{equation}
Taking $\liminf_{t\to 0^+}$ and applying the Fatou lemma we obtain $W_{p,\delta d_W/p}(f)\leq C\|f\|_{p,1/p}$ as desired.
\end{proof}

In the course of the proof we established the following locality-in-time estimate for the Besov norm which will be useful later.
\begin{corollary}\label{cor-Pprop}
Let $1\leq p<\infty$. There exists $C>0$ such that for any $f\in \B^{p,1/p}(X)$
\begin{equation*}
\|f\|_{p,1/p}\leq C \liminf_{t\to0^+} \frac1{t^{1/p}} \int_X\int_X|f(x)-f(y)|^pp^{(\delta)}_t(x,y)\,d\mu(x)\,d\mu(y).
\end{equation*}
\end{corollary}
This condition was previously considered in~\cite[Definition~6.7]{ABCRST1}, where it was called property $(P_{p,1/p})$.

\begin{remark}
In the case $X=\mathbb{R}^d$, we have $d_W=2$ and the fractional Sobolev space $\mathcal{W}^{2\delta/p,p}(X)$ coincides with the usual Euclidean fractional Sobolev space from~\cite{Gag57}, which is defined as
\begin{equation*}
W^{2\delta/p,p}(\mathbb{R}^d):=\bigg\{f\in L^p(\mathbb{R}^d,dx)\;\colon\; \bigg(\int_{\mathbb{R}^d}\int_{\mathbb{R}^d}\frac{|f(x)-f(y)|^p}{|x-y|^{d+2\delta}}dy\,dx\bigg)^{1/p}<\infty\bigg\}.
\end{equation*}
\end{remark}

\subsection{The case $\alpha >1/p$}
The spaces $\B^{p,\alpha}(X)$ for $\alpha>1/p$ are trivial and thus not interesting for further analysis. This completes our classification of these spaces.

\begin{proposition}\label{trivial Lp}
Let $1\leq p<\infty$ and $\alpha >1/p$. Then, $\B^{p,\alpha}(X)$ only contains the zero function.
\end{proposition}

\begin{proof}
The estimate~\eqref{E:W_norm_02} gives 
\begin{equation*}
\int_X\int_X\frac{|f(x)-f(y)|^p}{(t^{\frac{1}{\delta d_W}}+d(x,y))^{d_H+\delta d_W}}d\mu(y)\,d\mu(x) 
\leq  C t^{\alpha p-1} \| f \|_{p,\alpha}.
\end{equation*}
Applying Fatou's lemma to the $\liminf_{t\to0^+}$ we conclude from $\alpha p>1$ that 
\[
\int_X\int_X\frac{|f(x)-f(y)|^p}{d(x,y)^{d_H+\delta d_W}}d\mu(y)\,d\mu(x) =0
\]
which implies that $f$ is constant and thus zero since $f$ is in $L^p$.
\end{proof}

\subsection{Comparison to Besov metric spaces previously considered in the literature}

One can also compare the space $\mathbf{B}^{p,\alpha}(X)$ to Besov type metric spaces previously 
considered by Grigor'yan in~\cite{Gri03}. 
To do so, we define,
for any $f\in L^{p}(X,\mu)$ and $r>0$, 
\begin{equation}\label{eq:Besov-seminorm-r}
N^{\alpha}_{p}(f,r):=\frac{1}{r^{\alpha+d_{H}/p}}\biggl(\iint_{\{d(x,y)<r\}}|f(x)-f(y)|^{p}\,d\mu(x)\,d\mu(y)\biggr)^{1/p}.
\end{equation}
Furthermore, for any $\max\{1,p\}\leq q<\infty$ we set
\begin{equation}\label{eq:Besov-seminorm-pq}
N_{p,q}^{\alpha}(f):=\brak{\int_0^\infty \brak{N_p^{\alpha}(f,r)}^{q}\, \frac{dr}{r}}^{1/q},
\end{equation}
and for $q=\infty$ define
\begin{equation}\label{eq:Besov-seminorm}
N^{\alpha}_{p,\infty}(f):=\sup_{r>0}N^{\alpha}_{p}(f,r).
\end{equation}
The Besov metric space $\mathfrak{B}^{\alpha}_{p,q}(X)$, see~\cite{Gri03}, is defined as
\begin{equation}\label{eq:Besov}
\mathfrak{B}^{\alpha}_{p,q}(X):=\bigl\{f\in L^{p}(X,\mu)\,\colon\, N_{p,q}^{\alpha}(f)<\infty\bigr\}.
\end{equation}

With these notations, it is not difficult to rewrite the results of this section as follows:
\begin{proposition}\label{Besov non-local 2}
Let $p \ge 1$.
\begin{enumerate}[leftmargin=1.75em,label=\rm (\roman*)]
\item  If $ 0\le \alpha < 1/p$ we have $\mathbf{B}^{p,\alpha}(X)=\mathfrak{B}^{\alpha \delta d_W}_{p,\infty}(X)$ and  $\| f \|_{p,\alpha} \simeq N_{p,\infty}^{\alpha d_W}(f)$.
\item  $\mathbf{B}^{p,1/p}(X)=\mathfrak{B}^{\delta d_W/p}_{p,p}(X)$ and $\| f \|_{p,1/p } \simeq N_{p,p}^{ \delta d_W/p}(f)$.
\end{enumerate}
\end{proposition}

\section{BV functions and fractional Laplacian}\label{S:BVandFL}

In this section we introduce and analyze the space of BV functions that arises as the space $\mathbf{B}^{1,\alpha}(X)$ at the critical value of the exponent $\alpha$. The weak Bakry-\'Emery estimate from Assumption~\ref{A3}  plays a crucial role when this critical exponent is less than $1$, because then the results rely upon those obtained in~\cite{ABCRST3} for the local setting.  Indeed, many of our results in this case are easily proved from the equivalence of seminorms in Proposition~\ref{T:KS-Besov} and the characterizations of Korevaar-Schoen spaces in~\cite{ABCRST3}.  We will also see that new reasoning is needed to analyze the weak Sobolev spaces that occur in the alternative situation where the critical exponent is $\alpha=1$.

\subsection{$L^1$ critical exponent}
Recall that $0<\kappa<d_W$ is the H\"older regularity parameter from the weak Bakry-\'Emery condition~\eqref{E:wBECD} and we defined in the introduction the critical exponent
\begin{align*}
&\alpha_1^{\#}=\sup\{\alpha>0\,\colon\,\mathbf{B}^{1,\alpha}(X)\text{ contains non-constant functions}\}.\label{E:alpha_moll} 
\end{align*}

\begin{theorem}\label{L1 critical}\mbox{}
One has the following:
\begin{enumerate}[leftmargin=1.75em,label=\rm (\roman*)]
\item If $\delta \le 1-\frac{\kappa}{d_W}$, then $\alpha_1^{\#}=1$.
\item If $\delta > 1-\frac{\kappa}{d_W}$, then $\alpha_1^{\#}=\frac{1}{\delta} \left( 1-\frac{\kappa}{d_W}\right) $.
\end{enumerate}
\end{theorem}

\begin{proof}
Notice first that Proposition~\ref{trivial Lp} implies $\alpha_1^\#\leq 1$. In addition, for any $\alpha <1$, Theorem~\ref{T:KS-Besov} yields
\begin{equation}\label{E:critical_KS}
\B^{1,\alpha}(X)=KS^{\alpha \delta d_W,1}(X)=\mathcal{KS}^{\alpha \delta d_W,1}(X).
\end{equation}
In view of~\cite[Theorem 4.9]{ABCRST3}, the critical exponent in the Korevaar-Schoen space is $d_W-\kappa$, hence $\alpha_1^\#=\min\Bigl\{1,\frac{1}{\delta}\big(1-\frac{\kappa}{d_W}\big)\Bigr\}$.
\end{proof}

\begin{example}
If $(\mathcal{E}^{(\delta)},\mathcal{F}^{(\delta)})$, $\delta \in (0,1)$, is the Dirichlet form on $\mathbb{R}^n$ associated with the fractional Laplacian $(-\Delta)^\delta$, namely
\[
\mathcal{E}^{(\delta)}(f,f) \simeq  \int_{\mathbb{R}^n}\int_{\mathbb{R}^n} \frac{|f(x)-f(y)|^2}{\| x-y\|^{n+2\delta }}\,dy\,dx\,,
\]
then $\kappa=1$ and $d_W=2$, so $\alpha^\#_1=\min \big\{1, \frac{1}{2\delta}\big\}$ and
\begin{align*}
\mathbf{B}^{1,\alpha^\#_1}(X) =
\begin{cases}
\mathcal{KS}^{1,1}(\mathbb{R}^n)=\mathbf{BV}(\mathbb{R}^n) & \text{ if } \delta>1/2, \\
\mathcal W^{2\delta,1} (\mathbb{R}^n) &   \text{ if } \delta <1/2.
\end{cases}
\end{align*}

Interestingly, for $\delta=\frac{1}{2}$, one has
\[
\mathbf{B}^{1,\alpha^\#_1}(X) =\left\{ f \in L^1(\mathbb R^n,dx)\,\colon\,  \int_{\mathbb R^n}\int_{\mathbb R^n}\frac{|f(x)-f(y)|}{|x-y|^{n+1}}dy\,dx <+\infty \right\}
\]
which, by~\cite[Proposition 1]{Brezisconstant}, is a trivial space containing only the zero function. 
\end{example}

\subsection{BV functions}

In~\cite{ABCRST3} it was argued that Korevaar-Schoen and Besov spaces provide natural analogues of the space of bounded variation functions in certain metric settings.  The variation was defined as follows.

\begin{definition}\label{def-BV}
Set $BV(X):=KS^{d_W-\kappa,1}(X)$ and for $f \in BV(X)$ let
\[
\mathbf{Var} (f):=\liminf_{r\to 0^+}\int_X\int_{B(x,r)}\frac{|f(y)-f(x)|}{r^{d_W-\kappa} \mu(B(x,r))}\, d\mu(y)\, d\mu(x).
\]
\end{definition}

Observe that the crucial difference between this and the $KS^{d_W-\kappa,1}(X)$ norm is that $\mathbf{Var} (f)$ has a $\liminf$ rather than a $\limsup$.  The fact that the $\liminf$ is sufficient for the theory at the critical exponent $\alpha=\alpha_1^\#$ when  $\alpha_1^\#<1$ (equivalently $\delta > 1-\frac{\kappa}{d_W}$) 
relies heavily on the weak Bakry-\'Emery estimate through the results of~\cite{ABCRST3}. The corresponding results when $\alpha_1^\#=1$ use the much easier Corollary~\ref{cor-Pprop}.  One major consequence of the weak Bakry-\'Emery estimate is the following characterization of BV functions in terms of the semigroup-defined space $\mathbf{B}^{1,\alpha_1^\#}(X)$ at the critical exponent.

\begin{theorem}\label{T:BV_char}
Assume 
$\alpha_1^\#<1$.
Then $\mathbf{B}^{1,\alpha_1^\#}(X) =BV(X)$ and there exist constants $c,C>0$ such that for every $f \in BV(X)$,
\[
c \mathbf{Var} (f) \le \| f \|_{1,\frac{1}{\delta} \left( 1-\frac{\kappa}{d_W}\right)} \le C \mathbf{Var} (f).
\]
\end{theorem}
\begin{proof}
According to~\cite[Proposition~4.1, Theorem~4.9 ]{ABCRST3}, the weak Bakry-\'Emery estimate implies that 
the Korevaar-Schoen norm $\|f\|_{KS^{d_W-\kappa,1}}$ is bounded above and below by $\mathbf{Var} (f)$.  Since $\alpha_1^\#<1$ we can apply the equivalence of seminorms established in Theorem~\ref{T:KS-Besov} to obtain the result.
\end{proof}

\begin{remark}
In the case $\delta \le 1-\frac{\kappa}{d_W}$, Theorem~\ref{L1 critical} yields the critical parameter $\alpha_1^\#=1$ and we have that $\mathbf{B}^{1,1}(X)=\mathcal{W}^{ \delta d_W,1}(X)$ is a fractional Sobolev space. In this regime, the variation of a function is the non local-quantity
\[
\liminf_{t \to 0^+} \frac{1}{t} \int_X \int_X |f(x)-f(y) | p^{(\delta)}_t (x,y) d\mu(x) d\mu(y) \simeq \int_X\int_X\frac{|f(x)-f(y)|}{d(x,y)^{d_H+\delta d_W}}\,d\mu(y)d\mu(x) .
\]
When $X=\mathbb{R}^d$, one has $d_H=d$, $d_W=2$, $\kappa=1$ and the above notion of variation of a function related to the fractional Laplacian $(-\Delta)^\delta$, $\delta \le 1/2$, coincides with the notion of fractional variation and associated fractional perimeter extensively studied in relation to the theory of non-local minimal surfaces, see for instance \cite{ADM11,BFV14,CRS10,CV11, FFMMM, FLS08} and the references therein.
\end{remark}

\subsection{Co-area formulas}\label{ssec-coarea}
In this section, we prove that the Besov norm $\| \cdot \|_{1,\alpha}$ always satisfies co-area type estimates at the critical exponent $\alpha=\alpha_1^\#$.
In view of Theorem~\ref{L1 critical} we must consider the cases $\alpha_1^\#<1$ and $\alpha_1^\#=1$.  In both we write $E_t(f):=\{x\in X\,\colon\,f(x)>t\}$.
\begin{theorem}
Assume 
$\alpha_1^\#<1$. 
There exist constants $C_1,C_2>0$ such that for any non-negative $f\in L^1(X,\mu)$ and $t>0$, 
\begin{equation*}
C_1\int_0^\infty  \| \mathbf{1}_{E_t(f)} \|_{1,\alpha_1^\#} \,dt\leq  \| f \|_{1,\alpha_1^\#}  \leq C_2\int_0^\infty \| \mathbf{1}_{E_t(f)} \|_{1,\alpha_1^\#} \,dt.
\end{equation*}
In particular, $\mathbf{1}_{E_t(f)}\in \mathbf{B}^{1,\alpha_1^\#}(X) $ for any $f\in \mathbf{B}^{1,\alpha_1^\#}(X) $ and almost every $t>0$. Conversely, if $\mathbf{1}_{E_t(f)}\in \mathbf{B}^{1,\alpha_1^\#}(X) $ for almost every $t>0$ and $\int_0^\infty\| \mathbf{1}_{E_t(f)} \|_{1,\alpha_1^\#} \,dt<\infty$, then $f\in \mathbf{B}^{1,\alpha_1^\#}(X) $.
\end{theorem}
\begin{proof}
By virtue of Theorem~\ref{T:BV_char}, this follows from~\cite[Theorem 4.15]{ABCRST3}.
\end{proof}

When $\alpha_1^\#=1$ (equivalently $\delta \le 1-\frac{\kappa}{d_W}$), the critical Besov space is $\mathbf{B}^{1,1}(X)=\mathcal{W}^{ \delta d_W,1}(X)$, so the corresponding result is proved by a different argument.

\begin{theorem}\label{coarea 2}
Assume $\alpha_1^\#=1$.
 There exist constants $C_1,C_2>0$ such that for any non-negative $f\in L^1(X,\mu)$ and $t>0$
\begin{equation*}
C_1\int_0^\infty  \| \mathbf{1}_{E_t(f)} \|_{1,1} \,dt\leq  \| f \|_{1,1}  \leq C_2\int_0^\infty \| \mathbf{1}_{E_t(f)} \|_{1,1} \,dt.
\end{equation*}
In particular, $\mathbf{1}_{E_t(f)}\in \mathbf{B}^{1,1}(X) $ for any $f\in \mathbf{B}^{1,1}(X) $ and almost every $t>0$. Conversely, if $\mathbf{1}_{E_t(f)}\in \mathbf{B}^{1,1}(X) $ for almost every $t>0$ and $\int_0^\infty\| \mathbf{1}_{E_t(f)} \|_{1,1}\,dt<\infty$, then $f\in \mathbf{B}^{1,1}(X)$.
\end{theorem}

\begin{proof}
Since $f \ge 0$, for $\mu$-almost every $x,y \in X$ we can write
\begin{equation*}
| f(y)-f(x)| =\int_{0}^{+\infty} |\mathbf{1}_{E_t(f)}(x) -\mathbf{1}_{E_t(f)}(y) | dt.
\end{equation*}
Therefore,
\begin{align*}
W_{\delta d_W, 1}(f)
=\int_X\int_X\frac{|f(x)-f(y)|}{d(x,y)^{d_H+\delta d_W}}\,d\mu(y)d\mu(x) 
=\int_X\int_X \int_0^{+\infty} \frac{|\mathbf{1}_{E_t(f)}(x) -\mathbf{1}_{E_t(f)}(y) |}{d(x,y)^{d_H+\delta d_W}}\,d\mu(y)d\mu(x)
\end{align*}
and the result follows from Fubini's theorem and Proposition \ref{P:W_norm}.
\end{proof}


\subsection{$L^1$ pseudo-Poincar\'e inequality}

The pseudo-Poincar\'e inequalities introduced in~\cite{S-C92,CS-C93} are a useful tool to prove Sobolev inequalities, see e.g.~\cite{S-C10}. In the present setting, we obtain the following ones.

\begin{theorem}\label{T:pseudo_PI}
Assume 
$\alpha_1^\#<1$. 
Then, for every  
 $ f \in \mathbf{B}^{1,\alpha_1^\#}(X)$
\[
\| P^{(\delta)}_t f -f \|_{L^1(X,\mu)} \le C t^{\alpha_1^\#} \mathbf{Var} (f).
\]
\end{theorem}

\begin{proof}
Applying~\eqref{eq:Ptdelta} and~\cite[Proposition 3.10]{ABCRST3} we get
\begin{align*}
\|P_t^{(\delta)}f-f\|_{L^1(X,\mu)}
&=\int_X\Big|\int_0^\infty \eta_t^{(\delta)}(s)P_s(f-f(x))(x)\,ds\,\Big|\,d\mu(x)\\
&\leq
\int_0^\infty\eta_t^{(\delta)}(s)\|P_sf-f\|_{L^1(X,\mu)}\,ds\\
&\leq C\int_0^\infty\eta_t^{(\delta)}(s)\,s^{1-\frac{\kappa}{d_W}}ds\bigg(\liminf_{\tau\to 0^+}\frac{1}{\tau^{1-\frac{\kappa}{d_W}}}\int_XP_\tau(|f-f(x)|)(x)\,d\mu(x)\bigg).
\end{align*}
In view of~\cite[Lemma 4.13]{ABCRST3} and using~\eqref{moment eta}, this is bounded by $\frac{C\Gamma(1-\alpha_1^\#)}{\Gamma(\kappa/d_W)}t^{\alpha_1^\#}\mathbf{Var}(f)$.
\end{proof}

\begin{theorem}
Assume 
$\alpha_1^\#=1$.  
Then, for every $ f \in \mathbf{B}^{1,1}(X)$,
\[
\| P^{(\delta)}_t f -f \|_{L^1(X,\mu)} \le C t \int_X\int_X\frac{|f(x)-f(y)|}{d(x,y)^{d_H+\delta d_W}}\,d\mu(y)d\mu(x).
\]
\end{theorem}

\begin{proof}
Since 
\begin{equation*}
\| P^{(\delta)}_t f -f \|_{L^1(X,\mu)}\leq \int_X\int_Xp_t^{(\delta)}(x,y)|f(x)-f(y)|\,d\mu(y)\,d\mu(x),
\end{equation*}
the assertion follows as in~\eqref{E:W_norm_01}.
\end{proof}


\subsection{Isoperimetric  inequalities}

In this section we prove a Sobolev type inequality for the fractional Sobolev space $\mathcal{W}^{ \delta d_W,1}(X)$  and hence a fractional isoperimetric inequality.  For this purpose we must restrict to the case  $\delta \le 1-\frac{\kappa}{d_W}$, because if $\delta> 1-\frac{\kappa}{d_W}$ then, by Theorem~\ref{L1 critical}, $\alpha_1^\#<1$ and hence (using Proposition~\ref{P:W_norm}) $\mathcal{W}^{ \delta d_W,1}(X)=\mathbf{B}^{1,1}(X)$ is the space of constant functions.  The fact that we identify this critical range for $\delta$ also distinguishes our result from the corresponding one that appears, with a different proof,  in~\cite[Theorem 9.1]{BCLS95}.

\begin{theorem}
Assume $d_H > \delta d_W$. There exists a constant $C>0$ such that for every $f \in \mathbf{B}^{1,1}(X)$,
\[
\| f \|_{L^q(X,\mu)} \le C \int_X\int_X\frac{|f(x)-f(y)|}{d(x,y)^{d_H+\delta d_W}}\,d\mu(y)d\mu(x) ,
\]
where $q=\frac{d_H}{ d_H-\delta d_W}$.
\end{theorem}

\begin{proof}
Since Corollary~\ref{cor-Pprop} establishes the property $(P_{1,1})$ defined in~\cite[Definition 6.7]{ABCRST1}, the assertion follows from~\cite[Theorem 6.9]{ABCRST1} and Proposition~\ref{P:W_norm}.
\end{proof}

As a corollary, one deduces the following fractional isoperimetric inequality, which is the global analogue in our  setting of the known  fractional relative isoperimetric inequalities noted  in a Euclidean setting in~\cite{FLS08}, see also \cite{FFMMM}.

\begin{corollary}
 Assume $d_H > \delta d_W$. There exists a constant $C>0$ such that for every $E \subset X$ with finite measure,
\[
\mu (E)^\frac{ d_H-\delta d_W}{d_H} \le C \int_E\int_{X \setminus E}\frac{1}{d(x,y)^{d_H+\delta d_W}}\,d\mu(y)d\mu(x).
\]
\end{corollary}

In the case $d_H =\delta d_W$ the situation is different.

\begin{proposition}\label{Sobolev 2}
Assume $d_H =\delta d_W$. Then, $\mathbf{B}^{1,1}(X) \subset L^\infty(X,\mu)$ and there exists a constant $C>0$ such that for every $f \in\mathbf{B}^{1,1}(X)$ and almost every $u,v \in X$,
\[
| f(u)-f(v)| \le C \int_X\int_X\frac{|f(x)-f(y)|}{d(x,y)^{d_H+\delta d_W}}\,d\mu(y)d\mu(x).
\]
\end{proposition}

\begin{proof}
Let $f \in \mathbf{B}^{1,1}(X)$. Without loss of generality, we  assume $f \ge 0$ almost everywhere. For almost every $t \ge0 $ we define the set $E_t(f)=\{x\in X\,\colon f(x)>t\}$ as in Section~\ref{ssec-coarea}. Since $d_H =\delta d_W$, according to \cite[Corollary 6.6]{ABCRST1}, there is $c>0$ such that for every set $E$ of positive measure satisfying $\| 1_E \|_{1,1}<+\infty$, one has $\| 1_E \|_{1,1} \ge c$. 
However, from Theorem \ref{coarea 2}, there is $C>0$, such that
\[
\int_0^\infty  \| \mathbf{1}_{E_t(f)} \|_{1,1} \,dt\leq   C \| f \|_{1,1}  <+\infty.
\]
Therefore, the set $\Sigma (f):=\{t>0\,\colon\,\mu( E_t(f)) >0\}$ has finite Lebesgue measure. 
and from Fubini's theorem we obtain
\[
\int_X \int_{\mathbb{R} \setminus \Sigma_f} \mathbf{1}_{E_t(f)}(x) dt d\mu(x) =\int_{\mathbb{R} \setminus \Sigma_f} \mu( E_t(f) ) dt =0
\]
and hence $\int_{\mathbb{R} \setminus \Sigma_f} \mathbf{1}_{E_t(f)}(x) dt=0$ $\mu$-almost everywhere.
Thus, for $\mu$-almost every $u,v \in X$,
\begin{align*}
| f(y)-f(x)| 
=\int_{0}^{+\infty} |\mathbf{1}_{E_t(f)}(x) -\mathbf{1}_{E_t(f)}(y) | dt 
&=\int_{\Sigma (f)} |\mathbf{1}_{E_t(f)}(x) -\mathbf{1}_{E_t(f)}(y) | dt \\
 & \le \frac{1}{c} \int_{\Sigma (f)} \| \mathbf{1}_{E_t(f)} \|_{1,1} dt 
 \le \frac{C}{c} \| f \|_{1,1}
\end{align*}
at which point we apply Proposition~\ref{P:W_norm}.
\end{proof}

\section{$L^p$ theory}\label{S:fractional_Lp}
In the previous section the space of BV functions has been shown to correspond with the space $\mathbf{B}^{1,\alpha}(X)$ at the critical parameter $\alpha=\alpha_1^\#$, c.f.\ Theorem~\ref{L1 critical} and Theorem~\ref{T:BV_char}. In the same spirit, this section analyzes the spaces $\mathbf{B}^{p,\alpha}(X)$ for a generic $p\geq 1$ at their corresponding critical exponent. These will turn out to coincide with fractional Sobolev spaces.

\subsection{$L^p$ critical exponents}
Following the same notation as in~\cite{ABCRST1,ABCRST3}, we define for any  $1\le p<\infty$ the critical exponent 
\begin{align*}
&\alpha_p^{\#}=\sup\{\alpha>0\,\colon\,\mathbf{B}^{p,\alpha}(X)\text{ contains non-constant functions}\}\label{E:alpha_moll_p} 
\end{align*}
and the parameter
\begin{equation*}\label{eq:defnofbetap}
\beta_p= \Bigl(1-\frac{2}{p}\Bigr)\frac{\kappa}{d_W}+\frac{1}{p},
\end{equation*}
where $\kappa>0$ is the H\"older regularity parameter from the Bakry-\'Emery condition~\eqref{E:wBECD}.

\begin{theorem}\mbox{}
One has the following:
\begin{enumerate}[leftmargin=1.75em,label=\rm (\roman*)]
\item If $1 \le p < 2$  then $\frac{1}{2\delta}\leq\alpha_p^\#\leq \min \{ \frac{\beta_p}{\delta} , \frac{1}{p} \} $.
\item If $p \ge 2$ then $\alpha_p^{\#}=\frac{1}{p} $.
\end{enumerate}
\end{theorem}

\begin{proof}
We first notice that from Proposition~\ref{trivial Lp}, one always has $\alpha_p^{\#} \le \frac{1}{p}$. In particular, Theorem~\ref{T:KS-Besov} applies as in the proof of Theorem~\ref{L1 critical} and now~\cite[Theorem 3.11]{ABCRST3} yields $\frac{1}{2\delta}\leq\alpha_p^\#\leq\frac{\beta_p}{\delta}$ when $1\leq p<2$.  
In the case $p\geq 2$, combining Corollary~\ref{C:Lp_dense} with the fact that $\mathbf{B}^{p,1/p}(X)\subseteq \mathbf{B}^{p,\alpha}(X)$ for any $\alpha<\frac{1}{p}$, see~\cite[Lemma 4.1]{ABCRST1}, it follows that $\mathbf{B}^{p,\alpha}(X)$ is dense in $L^p$ and in particular non-trivial for any $\alpha<\frac{1}{p}$, hence $\alpha_p^\#=\frac{1}{p}$.
\end{proof}

\subsection{$\B^{p,1/p}$ is dense in $L^p$ for $p \ge 2$.}
As a consequence of the previous lemma, we can now use the space $\B^{p,1/p}(X)$  to analyze the fractional Sobolev space $\mathcal{W}^{\delta d_W/p,p}(X)$. In particular, we will show this space to be dense in $L^p$.

\begin{lemma}\label{L:Lp-Linf-cont}
Let $1\leq p<\infty$. There exists $C>0$ such that for any $t>0$ and $f\in L^p(X,\mu)$ 
\begin{equation*}
\|P^{(\delta)}_tf\|_{L^\infty(X,\mu)}\leq\frac{C}{t^{d_H/p\delta d_W}}\|f\|_{L^p(X,\mu)},
\end{equation*}
i.e.\ the operator $P^{(\delta)}_t\colon L^p(X,\mu)\to L^\infty(X,\mu)$ is bounded.
\end{lemma}
\begin{proof}
Since the semigroup $P_t$ is conservative, H\"older's inequality and the upper heat kernel bound yield
\begin{align*}
|P^{(\delta)}_tf(x)|^p&\leq
\int_X |f(y)|^pp^{(\delta)}_t(x,y)\,d\mu(y)\leq \frac{C}{t^{d_H/\delta d_W}}\|f\|_{L^p(X,\mu)}^p.\qedhere
\end{align*}
\end{proof}

\begin{proposition}
Let $ p\geq 2$. There exists $C>0$ such that for any $t>0$ and $f\in L^p(X,\mu)$ 
\begin{align*}
\| P^{(\delta)}_t f \|_{p,1/p} \le \frac{C}{t^{1/p}} \| f \|_{L^p(X,\mu)}
\end{align*}
\end{proposition}

\begin{proof}
By virtue of Proposition~\ref{P:W_norm} and Lemma~\ref{L:Lp-Linf-cont} we have
\begin{align*}
\| P^{(\delta)}_t f \|^p_{p,1/p} & \le C \int_X\int_X\frac{|P^{(\delta)}_t f(x)-P^{(\delta)}_t f(y)|^p}{d(x,y)^{d_H+\delta d_W}} \,d\mu(y)d\mu(x) \\
 & \le  C \int_X\int_X ( |P^{(\delta)}_t f(x)-P^{(\delta)}_t f(y)|^{p-2})\frac{|P^{(\delta)}_t f(x)-P^{(\delta)}_t f(y)|^2}{d(x,y)^{d_H+\delta d_W}} \,d\mu(y)d\mu(x) \\
 & \le C  \| P^{(\delta)}_t f \|_{L^\infty(X,\mu)}^{p-2}\int_X\int_X \frac{|P^{(\delta)}_t f(x)-P^{(\delta)}_t f(y)|^2}{d(x,y)^{d_H+\delta d_W}} \,d\mu(y)d\mu(x)  \\
 & \le \frac{C}{t^{d_H(p-2)/p\delta d_W}}\|f\|^{p-2}_{L^p(X,\mu)} \int_X\int_X \frac{|P^{(\delta)}_t f(x)-P^{(\delta)}_t f(y)|^2}{d(x,y)^{d_H+\delta d_W}} \,d\mu(y)d\mu(x) \\
 & \le \frac{C}{t^{d_H(p-2)/p\delta d_W}}\|f\|^{p-2}_{L^p(X,\mu)} \mathcal{E}^{(\delta)} (P^{(\delta)} _tf,P^{(\delta)} _tf) \\
 & \le\frac{C}{t^{1 +d_H(p-2)/p\delta d_W}}\|f\|^{p-2}_{L^p(X,\mu)} \| f \|^2_{L^2(X,\mu)}.
\end{align*}
Let us now fix a compact set $K\subset X$. The latter inequality applied to $f\mathbf{1}_K$ and H\"older's inequality yield
\begin{equation*}\label{E:PtinBesov_h1}
\| P^{(\delta)}_t f\mathbf{1}_K \|_{p,1/p}\le\frac{C^{1/p}}{t^{\frac{1}{p} +\frac{d_H(p-2)}{p^2\delta d_W}}}\|f\|^{\frac{p-2}{p}}_{L^p(K,\mu)} \| f \|^{\frac{2}{p}}_{L^2(K,\mu)}\leq \frac{C^{1/p}\mu(K)^{\frac{p-2}{p}}}{t^{\frac{1}{p} +\frac{d_H(p-2)}{p^2\delta d_W}}}\|f\|_{L^p(X,\mu)}.
\end{equation*}
Letting $p\to\infty$, it follows that, for any $s>0$, the operator $\mathcal{P}_t^{(\delta)}\colon L^\infty(X,\mu)\to L^\infty(K\times X,p^{(\delta)}_s\mu\otimes\mu)$ defined as $\mathcal{P}_t^{(\delta)}f(x,y)=P_t^{(\delta)}f(x)-P_t^{(\delta)}f(y)$ is bounded by $1$.  Because the bound does not depend on $K$, the same is true for $\mathcal{P}_t^{(\delta)}\colon L^\infty(X,\mu)\to L^\infty(X\times X,p^{(\delta)}_s\mu\otimes\mu)$. On the other hand, we know from~\cite[Theorem 5.1]{ABCRST1} that
\begin{equation*}
\|P_t^{(\delta)}f\|_{2,1/2}\leq \frac{C}{t^{1/2}}\|f\|_{L^2(X,\mu)},
\end{equation*}
i.e. $\mathcal{P}^{(\delta)}_t\colon L^2(X,\mu)\to L^2(X\times X,p_s^{(\delta)}\mu\otimes\mu)$ is bounded by $C(s/t)^{1/2}$. The Riesz-Thorin interpolation theorem now yields that $\mathcal{P}^{(\delta)}_t\colon L^p(X,\mu)\to L^p(X\times X,p_s\mu\otimes\mu)$ is bounded by 
$C(s/t)^{1/p}$, hence
\begin{equation*}
\frac{1}{s^{1/p}}\bigg(\int_X\int_Xp_t^{(\delta)}(x,y)|P_t^{(\delta)}f(x)-P_t^{(\delta)}f(y)|^pd\mu(y)\,d\mu(x)\bigg)^{1/p}\leq \frac{C}{t^{1/p}}\|f\|_{L^p(X,\mu)}.
\end{equation*}
Taking the supremum over $s>0$ on the left hand side yields the result.
\end{proof}
We  conclude by recording the noteworthy consequence of the previous proposition that was the main objective of this section.
\begin{corollary}\label{C:Lp_dense}
For $p \ge 2$, $\B^{p,1/p}(X)=\mathcal{W}^{ \delta d_W/p,p}(X)$ is dense in $L^p(X,\mu)$.
\end{corollary}

\begin{proof}
Let $ f\in L^p(X,\mu)$. Then for every $t>0$, $P_t f \in \mathcal{W}^{ \delta d_W/p,p}(X)$ and moreover, by $L^p$ strong continuity of the heat semigroup one has $\| P_t f - f \|_{L^p(X,\mu)} \to 0$ when $t \to 0$.
\end{proof}

\subsection{Sobolev inequalities}
The following Sobolev inequality is available in this setting.

\begin{theorem}
Let $p \ge 1$. Assume $d_H > \delta d_W$. There is  $C>0$ such that for every $f \in \mathbf{B}^{p,1/p}(X)$,
\[
\| f \|_{L^q(X,\mu)} \le C\left( \int_X\int_X\frac{|f(x)-f(y)|^p}{d(x,y)^{d_H+\delta d_W}}\,d\mu(y)d\mu(x)\right)^{1/p} ,
\]
where $q=\frac{pd_H}{ d_H-\delta d_W}$.
\end{theorem}

\begin{proof}
In Corollary~\ref{cor-Pprop} we saw that the condition $(P_{p,1/p})$  of~\cite[Definition 6.7]{ABCRST1} holds. By assumption $\frac{d_H}{\delta d_W}>1\geq \frac{1}{p}$, hence~\cite[Theorem 6.9]{ABCRST1} yields the desired inequality.
\end{proof}

\section{Putting things in perspective: A discussion on isoperimetric and Sobolev inequalities for non-local Dirichlet spaces}\label{S:final}
Although there are some versions of Poincar\'e type inequalities in the context of non-local Dirichlet forms, see e.g.~\cite{OS-CT08,CW17}, 
the standard ones do not make sense here.
In contrast, it is meaningful to ask for some global Sobolev type inequality: namely, whether there exist constants $C>0$ and $\kappa\ge 1$ such that
for every $f\in\mathcal{F}^{(\delta)}$,
\begin{equation}\label{E:GlobalSobolev}
\left(\int_X|f|^{2\kappa}\, d\mu\right)^{\frac{1}{2\kappa}}
  \le C \sqrt{\mathcal{E}^{(\delta)}(f,f)}.
\end{equation}
In this section we explore a condition under which the above inequality holds. 
Recall that a Markovian semigroup $\{P_t\}_{t>0}$ is called transient, see e.g.~\cite[p.38]{FOT94} if there exists an almost-everywhere
positive $f\in L^1(X,\mu)$ such that 
\[
\int_0^\infty P_tf\, dt<\infty
\]
$\mu$-a.e. in $X$, where $P_tf$ denotes as before the heat semigroup given by
\[
P_tf(x)=\int_X p_t(x,y)\, f(y)\, d\mu(y).
\]
Thus transience of the semigroup $\{P_tf\}_{t>0}$ thus implies that it decays fast as $t\to \infty$ almost everywhere in $X$. For the rest of this section we will assume that the semigroup $\{P^{(\delta)}_t\}_{t>0}$ is transient and strongly continuous. 
According to~\cite[Corollary 5.4]{bogdan03}, this happens if and only if 
\[
0<\delta<\min\{1, d_H/d_W\};
\]
see also~\cite{bogdan02} and~\cite[Remark 4.6]{Kum04}. Note that this condition is natural because it means that the transience is equivalent to the condition $d_S^{(\delta)}:=\frac{2d_W}{\delta d_H}>2$, where $d_S^{(\delta)}$ is the spectral dimension of the subordinated process. 
In this case, we have the following classical result that gives a first version of a Sobolev inequality.

\begin{lemma}[\protect{\cite[Theorem~1.5.3]{FOT94} or \cite[Theorem 2.1.5]{CF12}}]\label{lem:g-Sobolev}
Let $(\mathcal{E}^{(\delta)},\mathcal{F}^{(\delta)})$ be the Dirichlet form associated with $\{P^{(\delta)}_t\}_{t>0}$. The strongly continuous semigroup $\{P^{(\delta)}_t\}_{t>0}$ is 
transient if and only if 
there exists a bounded function $g\in L^1(X,\mu)$ that is strictly positive $\mu$-a.e. on $X$ such that 
\[
\int_X|f|\, g\, d\mu\le \sqrt{\mathcal{E}^{(\delta)}(f,f)}\qquad\forall\,f\in\mathcal{F}^{(\delta)}.
\]
\end{lemma}

While the conclusion of the above lemma looks remarkably like the desired Sobolev inequality~\eqref{E:GlobalSobolev}, the problem here is that we do not have good control over $g$. Moreover, we would like to estimate the $L^q(X,\mu)$-norm of $f$ for some $q>0$, not just the $L^1(X, g\, d\mu)$-norm. To circumvent this issue, we will introduce a capacitary condition on the underlying space $X$.

\medskip

Given a compact set $K\subset X$, one defines the variational capacity of $K$ as 
\[
\text{Cap}_0(K):=\inf\{ \mathcal{E}^{(\delta)}(f,f)\;\colon\;f\in\mathcal{F}^{(\delta)}\cap C_0(X)\text{ such that }f(x)\geq 1\;\forall\,x\in X\},
\]
where $C_0(X)$ denotes the collection of all compactly supported continuous functions on $X$, see~\cite[Section~2.4]{FOT94}. 
Under transience, the following capacitary type inequality holds.

\begin{lemma}[\protect{\cite[Lemma~2.4.1]{FOT94}}]
For any $f\in\mathcal{F}^{\delta}\cap C_0(X)$ and $t>0$
\[
\int_0^\infty 2t\, \text{\rm Cap}_0(E_t(|f|)\, dt\le 4\, \mathcal{E}^{(\delta)}(f,f),
\]
where $E_t(|f|):=\{x\in X\,\colon\,|f(x)|>t\}$.
\end{lemma}

The above lemma indicates that there are plenty of compact sets in $X$ with finite capacity. 
Given the above notion of capacity, the next theorem identifies a property on the Dirichlet form under which we have the desired Sobolev inequality.

\begin{theorem}[\protect{\cite[Theorem~2.4.1]{FOT94}}]\label{T:isoperim-nonlocal}
Assume that there is some $\kappa\ge 1$ and $\Theta>0$ such that 
\begin{equation}\label{eq:isoperim-nonlocal}
\mu(K)^{1/\kappa}\le \Theta\, \text{\rm Cap}_0(K)
\end{equation}
for any compact $K\subset X$. Then, there exists a constant $C>0$ such that
\begin{equation}\label{E:Sobolev-nonlocal}
\left(\int_X|f|^{2\kappa}\, d\mu\right)^{1/2\kappa}\le C\, \sqrt{\mathcal{E}^{(\delta)}(f,f)}
\end{equation}
for any $f\in\mathcal{F}^{(\delta)}$.
\end{theorem}

The inequality~\eqref{eq:isoperim-nonlocal} is an analog of the isoperimetric inequality adapted to the non-local Dirichlet form $(\mathcal{E}^{(\delta)},\mathcal{F}^{(\delta)})$.
The optimal constant $\Theta$ is called the isoperimetric constant of $\mathcal{E}^{(\delta)}$ and it is a non-local analog of the Cheeger constant.
Indeed, \cite[Theorem~2.4.1]{FOT94} claims even more, namely that~\eqref{eq:isoperim-nonlocal} is equivalent to a Sobolev type inequality like~\eqref{E:Sobolev-nonlocal}.
When such an inequality is available, the total capacity of a compact set $K\subset X$,
\[
\text{Cap}_1(K):=\inf\Big\{ \|f\|_{L^2(X,\mu)}^2+\mathcal{E}^{(\delta)}(f,f)~\colon~f\in\mathcal{F}\cap C_0(X)\text{ with }f(x)\geq 1\;\forall\,x\in K\Big\},
\]
has the same class of null capacity sets as $\text{Cap}_0$. We refer the interested reader to~\cite[Theorem~2.4.3]{FOT94} for examples of Sobolev type inequality for measures in Euclidean spaces.

\medskip

%
%

We point out here that this notion of isoperimetric inequality is weaker than the classical one, where $\text{Cap}_0(K)$ is replaced with the relative $1$-capacity of $K$ given by
\[
{\rm{Cap}}(K)=\inf\Big\{\int_X|\nabla f|\, d\mu~\colon~f\in \mathcal{F}\cap C_0(X)\text{ with }f(x)\ge 1\;\forall\,x\in K\Big\}.
\]
In the case of non-local Dirichlet forms the quantity $\int_X|\nabla f|\, d\mu$ is not defined in the usual sense, although a version of an integral of this type is available for arbitrary Dirichlet forms, c.f.~\cite{HRT13}. 

\bibliographystyle{amsplain}
\bibliography{BV6_Refs}

\

\ 

P.A.R.: \url{paruiz@math.tamu.edu}  \\
Department of Mathematics,
Texas A\&M University,
College Station, TX 77843

\

N.S.: \url{shanmun@uc.edu} \\
Department of Mathematical Sciences, 	
University of Cincinnati, 	
Cincinnati, OH 45221

\

F.B.: \url{fabrice.baudoin@uconn.edu}
L.C.: \url{li.4.chen@uconn.edu}
L.R.: \url{rogers@math.uconn.edu}
A.T.: \url{teplyaev@uconn.edu}\\
Department of Mathematics,
University of Connecticut,
Storrs, CT 06269

\end{document}